\documentclass[reqno,12pt]{amsart}
\usepackage{epsfig}
\usepackage[english]{babel}
\usepackage[latin1]{inputenc}
\usepackage{indentfirst}
\usepackage{graphicx,subfigure}
\usepackage{tabularx}
\usepackage{xspace} 
\usepackage{amsmath,amssymb}
\usepackage{color}
\definecolor{oneblue}{rgb}{0,0.0,0.75}

\usepackage{fancyhdr}
 \usepackage{setspace}

\makeatletter
\newcommand{\sech}{\mathop{\operator@font sech}}
\newcommand{\sign}{\mathop{\operator@font sign}}
\makeatother

\newtheorem{lemma}{Lemma}[section]
\newtheorem{theorem}{Theorem}[section]

\newtheorem{remark}{Remark}[section]
\numberwithin{equation}{section}

\begin{document}


\title[Error estimates for...]{Error estimates for semidiscrete Galerkin and collocation approximations to pseudo-parabolic problems with Dirichlet conditions}

\author{E. Abreu}
\address{Department of Applied Mathematics, IMECC, University of Campinas, Campinas, SP, Brazil.
Email: eabreu@ime.unicamp.br}

\author{A. Dur\'an}
\address{ Applied Mathematics Department,  University of
Valladolid, 47011 Valladolid, Spain. Email:angel@mac.uva.es}

%




\begin{abstract}
This paper is concerned with {the numerical approximation} of the Dirichlet initial-boundary-value problem of nonlinear pseudo-parabolic equations with spectral methods. Error estimates for the semidiscrete Galerkin and collocation schemes based on Jacobi polynomials are derived.
\end{abstract}
\maketitle

\section{Introduction}
The present paper deals with {the numerical analysis} to the initial-boundary-value problem (ibvp) of pseudo-parabolic type in $\Omega=(-1,1)$
\begin{eqnarray}
&&cv_{t}-(av_{xt})_{x}=-(\alpha v_{x})_{x}+\beta v_{x}+\gamma,\; x\in\Omega,\; t>0,\label{ad11a}\\
&&v(x,0)=v_{0}(x),\; x\in\Omega,\label{ad11b}\\
&&v(-1,t)=v(1,t)=0,\; t>0,\label{ad11c}
\end{eqnarray}
where $a=a(x), c=c(x)$ are  continuously differentiable functions in $\Omega$ and bounded above and below by positive constants. The right hand side of (\ref{ad11a}) involves linear and nonlinear terms
$\alpha=\alpha(x,t,v), \beta=\beta(x,t,v), \gamma=\gamma(x,t,v)$; they are assumed to be continuously differentiable functions of $x, t$ and $v$. 

{ In \cite{AbreuD2020}, we perform an extensive study, by computational means, of the use of spectral discretizations, of Galerkin and collocation type, based on Legendre and Chebyshev polynomials for models of type (\ref{ad11a})-(\ref{ad11c})}. The resulting semidiscrete systems are fully discretized there by suitable time integrators, with the aim at overcoming the possible midly stiff character of the ordinary differential problems and the difficulties to simulate nonsmooth data. This computational study is complemented by the present paper with a numerical analysis of the spectral discretization. More specifically, for Galerkin and collocation methods based on a family of Jacobi polynomials (which includes the Legendre and Chebyshev cases described in \cite{AbreuD2020}) existence of solution of the semidiscrete systems and error estimates in suitable Sobolev norms are derived. As usual for this kind of approaches, the results proved here establish the rate of convergence of the spectral approximation in terms of the regularity of the data of the problem. In particular, they justify those experiments in \cite{AbreuD2020} concerning spectral convergence in the smooth case.

Pseudo-parabolic equations of the form (\ref{ad11a}), in one or more dimensions, 
are used for modelling in different areas of Physics and Engineering. Relevant examples are the BBM-Burgers equation, a dissipative modification of the BBM equation for water waves, \cite{BBM1972}, and the pseudo-parabolic Buckley-Leverett equation describing two-phase flow in porous media, \cite{BearB1991}. We refer to the rich bibliography on it commented in \cite{AbreuD2020}.

The mathematical theory of pseudo-parabolic equations can be covered by \cite{MedeirosM1977,MedeirosP1977,ShowalterT1970,Showalter1978,FanP2011,BohmS1985,
CaoP2015,Colton1972,SeamV2011}. 
Existence and uniqueness of weak solutions to nonlinear pseudo-parabolic equations are proved in \cite{6PM}, whereas the existence of weak solutions for degenerate cases is studied in \cite{8MABH,9AM}. A homogenization of a closely related pseudo-parabolic system is considered in \cite{11PS}. Traveling wave solutions and their relation to non-standard shock solutions to hyperbolic conservation laws are investigated in \cite{12CH,13DPP} for linear higher order terms. Uniqueness of weak solutions for a pseudo-parabolic equations modelling flow in porous media can be found in \cite{ZC01,DKSL84,CaoP2015}. In \cite{GGG14}, the authors study existence and uniqueness of weak solutions of the initial and boundary value problem for a fourth-order pseudo-parabolic equation with variable exponents of non-linearity, along with a long-time behaviour of weak solutions. Finally, existence of weak solutions for a nonlocal pseudo-parabolic model for Brinkman two-phase flow model in porous media has been recently established, \cite{AJCR19}.
 
Concerning the numerical approximation of equations of the form (\ref{ad11a}), the literature contains many references involving finite differences, \cite{Amir1990,Amir1995,SunY2002,Amir2005,CuestaP2009,FanP2013}, as well as finite elements and finite volumes, \cite{Lu2017,Ning2017,AbreuV2017,Yang2008}. We also mention some convergence results. First, stability and convergence of difference approximations to pseudo-parabolic partial differential equations is discussed in \cite{FordT1973,FordT1974} and the time stepping Crank-Nicolson Galerkin method to approximate several nonlinear Sobolev-type problems is analyzed in \cite{Ewing1978,Ewing1975}.
Of particular relevance for the present study is the finite element approach for the nonlinear periodic-initial-boundary-value problem,  \cite{ArnoldDT1981}, where Arnold and collaborators obtain optimal error estimates, in $L^{2}$ and $H^{1}$norm, of a standard Galerkin method with continuous piecewise polynomials, and a nodal superconvergence. Moreover, Fourier spectral methods of Galerkin and collocation type for quasilinear pseudo-parabolic equations are analysed in \cite{Quarteroni1987}. This is, to our knowledge, the main reference about the use of spectral methods to approximate Sobolev equations. More recent convergence results can be found in  \cite{KPR17}, where an analysis of a linearization scheme for an interior penalty discontinuous Galerkin for a pseudo-parabolic model in porous media applications is considered. High-order finite differences are employed in \cite{Amir2017} and B-spline quasi-interpolation methods in \cite{RKSB16}. In addition, an adaptive mesh approach for pseudo-parabolic-type problems is introduced in \cite{ACHD20} and a Meshless RBFs method is considered  in \cite{HHG20}. Finally, unconditionally stable vector splitting schemes for pseudo-parabolic equations are constructed and analyzed in \cite{VabischevichG2013}. { It is worthwhile to mention that standard operator splitting may fail to capture the correct behavior of the solutions for pseudo-parabolic type differential models. In \cite{AbreuV2017}, the authors presented a non-splitting numerical method which is based on a fully coupled space-time mixed hybrid finite element/volume discretization approach to account for the delicate nonlinear balance between the hyperbolic flux and the pseudo-parabolic term linked to the full pseudo-parabolic differential model.}

The structure of the paper is as follows. Section \ref{sec2} is devoted to some theoretical aspects of (\ref{ad11a})-(\ref{ad11c}) as the weak formulation (already mentioned in \cite{AbreuD2020} for Legendre and Chebyshev cases) and assumptions on well-posedness. These preliminaries also include a summary on inverse inequalities, as well as projection and interpolation error estimates for the family of Jacobi polynomials under consideration. The contents of Section \ref{sec2} will be used to the numerical analysis of the spectral Galerkin approximation in Section \ref{sec3}, and the collocation approximation in Section \ref{sec4}. Both contain, under suitable hypotheses on the data of the problem, results on the existence of numerical solution and convergence to the solution of (\ref{ad11a})-(\ref{ad11c}). Concluding remarks and perspectives for future work are outlined in Section \ref{sec5}.

We now describe the main notation used throughout the paper. For positive integer $p$, $L^{p}(\Omega)$ denotes the normed space of $L^{p}$-functions on $\Omega$ with $||\cdot ||_{p}$ as associated norm, while for nonnegative integer $m$, $C^{m}(\overline{\Omega})$ is the space of $m$-th order continuously differentiable functions on $\overline{\Omega}=[-1,1]$. 

Let $-1<\mu<1$ and define the Jacobi weight function 
\begin{eqnarray}
w(x)=w_{\mu}(x)=(1-x^{2})^{\mu},\; x\in\Omega.\label{jacobiw}
\end{eqnarray}
($\mu=0$ corresponds to the Legendre case and $\mu={-1/2}$ to the Chebyshev case.) Then 
$L_{w}^{2}=L_{w}^{2}(\Omega)$ will denote the space of squared integrable functions with respect to the weighted inner product
\begin{eqnarray}
(\phi,\psi)_{w}=\int_{-1}^{1}\phi(x)\psi(x)w(x)dx,\; \phi,\psi\in L_{w}^{2},\label{12a}
\end{eqnarray}
and associated norm $||\phi||_{0,w}=(\phi,\phi)_{w}^{1/2}$. For the Sobolev spaces $H_{w}^{k}=H_{w}^{k}(\Omega), k\geq 0$ integer (where $H_{w}^{0}=L_{w}^{2}$) the corresponding norm will be denoted by
\begin{eqnarray*}
||\phi ||_{k,w}^{2}=\sum_{j=0}^{k}||\frac{d^{j}}{dx^{j}}\phi||_{0,j}^{2}.
\end{eqnarray*}
We will also consider the spaces $H_{w,0}^{k}=H_{w,0}^{k}(\Omega)$ of functions $\phi\in H_{w}^{k}$ such that $\phi(-1)=\phi(1)=0$. For $s\geq 0$, $H_{w}^{s}=H_{w}^{s}(\Omega)$ (and $H_{w,0}^{s}=H_{w,0}^{s}(\Omega)$) are defined by interpolation theory, \cite{Adams}. Note that in the case of the Legendre approximation ($w(x)=1$) the spaces $H_{w}^{s}, H_{w,0}^{s}$ are the standard Sobolev spaces $H^{s}, H_{0}^{s}$.

For an integer $N\geq 2$, $\mathbb{P}_{N}$ will stand for the space of polynomials of degree at most $N$ on $\overline{\Omega}$ and
\begin{eqnarray*}
\mathbb{P}_{N}^{0}=\{p\in\mathbb{P}_{N} / p(-1)=p(1)=0\}.
\end{eqnarray*}
If $T>0$ and $1\leq p\leq\infty$, $L^{p}(0,T)$ stands for the space of $L^{P}$ functions on $(0,T)$  with norm $|\cdot|_{p}$. For an integer $k\geq 0$, the space of $m$-th order continuously differentiable functions $u:(0,T]\rightarrow X$, where $X=H_{w}^{s}$ or $H_{w,0}^{s}, s\geq 0$, will be denoted by $C^{k}(0,T,X)$. Additionally, if $0<k<\infty$, $L^{k}(0,T,X)$ will stand for the normed space of functions $u:(0,T]\rightarrow X$ with associated norm
\begin{eqnarray*}
||u||_{L^{k}(0,T,X)}=\left(\int_{0}^{T}||u(t)||_{s,w}^{k}dt\right)^{1/k}.
\end{eqnarray*}
We also denote by $L^{\infty}(0,T,X)$ the space of functions $u:(0,T]\rightarrow X$ with finite norm
\begin{eqnarray*}
||u||_{L^{\infty}(0,T,X)}={\rm esssup}_{t\in (0,T)}||u(t)||_{s,w},
\end{eqnarray*}
where ${\rm essesup}$ stands for the essential spectrum.
Furthermore, $C^{1}(\Omega\times (0,T)\times \mathbb{R})$ (resp. $C_{b}^{1}=C_{b}^{1}(\Omega\times (0,T)\times \mathbb{R})$) will stand for the space of continuously differentiable (resp. uniformly bounded, continuously differentiable) functions $f=f(x,t,v)$ in $(x,t,v)\in \Omega\times (0,T)\times \mathbb{R}$.

The analysis of the collocation methods requires the introduction of discrete norms. Let $\{x_{j},w_{j}\}_{j=0}^{N}$ be the nodes and weights of the Gauss-Lobatto quadrature related to $w(x)$, \cite{Mercier,CanutoHQZ1988,BernardiM1997}. For $\phi, \psi$ continuous on $\overline{\Omega}$, the discrete inner product based on the Gauss-Lobatto data is denoted by
\begin{eqnarray}
\left(\phi,\psi\right)_{N,w}=\sum_{j=0}^{N}\phi(x_{j})\psi(x_{j})w_{j},\label{ad12}
\end{eqnarray}
with associated norm $||\phi||_{N,w}=\left(\phi,\phi\right)_{N,w}^{1/2}$. We recall that, \cite{CanutoHQZ1988}
\begin{eqnarray*}
\left(\phi,\psi\right)_{N,w}=\left(\phi,\psi\right)_{w},\label{ad13}
\end{eqnarray*}
if $\phi\psi\in\mathbb{P}_{2N-1}$. 
The equivalence of the norms $||\phi||_{N,w}$ and $||\phi||_{0,w}$ when $\phi\in \mathbb{P}_{N}$, established in the following lemma, was proved in \cite{CanutoQ1982a} for the case of Legendre and Chebyshev weights and in \cite{BernardiM1989} for (\ref{jacobiw}) with $\mu>-1$.
\begin{lemma}
\label{lemma11} Let $N\geq 2$ be an integer. Then there exist positive constants $C_{1}, C_{2}$, independent of $N$, such that for any $\phi\in \mathbb{P}_{N}$
\begin{eqnarray*}
C_{1}||\phi||_{0,w}\leq ||\phi||_{N,w}\leq C_{2} ||\phi||_{0,w}.\label{ad14}
\end{eqnarray*}
\end{lemma}

Finally $C$ will be used to denote a generic, positive constant, independent of $N$ and $u$, but that may depend on $t$ (this will be denoted by $C(t)$).
\section{Preliminaries}
\label{sec2}
\subsection{Weak formulation}
The analysis of the spectral discretizations that will be made below requires some hypotheses, properties and technical results concerning (\ref{ad11a})-(\ref{ad11c}) and the approximation in weighted norms. 
From now on we will fix $\mu\in (-1,1)$ and consider the weight (\ref{jacobiw}). 
The first property to be mentioned is the weak formulation of (\ref{ad11a})-(\ref{ad11c}), cf. \cite{AbreuD2020}
\begin{eqnarray}
A(v_{t}, \psi)=B(v,\psi),\; \psi\in H_{w,0}^{1} \label{ad25}
\end{eqnarray}
with $v(0)=v_{0}$ and
\begin{eqnarray}
A(\phi,\psi)&=&(c\phi,\psi)_{w}+L_{a}(\phi,\psi),\label{ad26}\\
B(\phi,\psi)&=&L_{\alpha}(\phi,\psi)+(\beta(\phi)\phi_{x},\psi)_{w}+(\gamma(\phi),\psi)_{w},\; \phi, \psi\in H_{w,0}^{1},\nonumber
\end{eqnarray}
where, for $d=d(x,t,v)$
\begin{eqnarray*}
L_{d}(\phi,\psi)=\int_{-1}^{1}d\phi_{x}(\psi w)_{x}dx.\label{ad24}
\end{eqnarray*}
Since $a$ is bounded above and below by positive constants, then $L_{a}$ is equivalent to
\begin{eqnarray}
L(\phi,\psi)=\int_{-1}^{1}\phi_{x}(\psi w)_{x}dx,\label{ad26a}
\end{eqnarray}
and therefore, \cite{BernardiM1989,CanutoHQZ1988,BernardiM1997}, the bilinear form $A$ in (\ref{ad26}) is continuous in $H_{w}^{1}\times H_{w,0}^{1}$ and elliptic in $H_{w,0}^{1}\times H_{w,0}^{1}$, that is, there are positive constants $C_{1}, C_{2}$ such that for all $\phi, \psi\in H_{w,0}^{1}$
\begin{eqnarray*}
|A(\phi,\psi)|&\leq &C_{1} \left(||\phi||_{0,w}||\psi||_{0,w}+||\phi_{x}||_{0,w}||\psi_{x}||_{0,w}\right)\\
&\leq &C_{1} ||\phi||_{1,w}||\psi||_{1,w},\; \phi\in H_{w}^{1}, \psi\in H_{w,0}^{1},\\
A(\phi,\phi)&\geq &C_{2} ||\phi||_{1,w}^{2},\;\phi, \psi\in H_{w,0}^{1}.
\end{eqnarray*}

The weak formulation (\ref{ad25}) is used to assume well-posedness of (\ref{ad11a})-(\ref{ad11c}), according to the following results, cf. \cite{ShowalterT1970,ArnoldDT1981}
\begin{theorem}
\label{theorem21} Let $T>0$ and assume that $a, c\in C^{1}(\Omega)$, $\alpha, \beta, \gamma \in C_{b}^{1}(\Omega\times (0,T)\times \mathbb{R})$. Given $v_{0}\in H_{w,0}^{1}$, then there is a unique solution $v\in C^{1}(0,T,H_{w,0}^{1})$ of (\ref{ad25}) with $||v||_{L^{\infty}(0,T,H_{w}^{1})}$ bounded by a constant depending only on $||v_{0}||_{1,w}$ and the data of the problem. Furthermore, if $v_{0}\in H_{w,0}^{k}$ with $k> 1$ integer, $a,c\in C^{k-1}({\Omega}), \alpha, \beta, \gamma, \delta\in C^{k-1}(\Omega\times (0,T)\times\mathbb{R})$,
then $v(t)\in H_{w,0}^{k}$ for all $t\in (0,T)$ and
\begin{eqnarray}
||v|||_{L^{\infty}(0,T,H_{w}^{k}) }+||v_{t}|||_{L^{\infty}(0,T,H_{w}^{k})} \leq C,\label{ad26b}
\end{eqnarray}
where $C$ is a constant depending only on $||v_{0}||_{k,w}$ and the data of the problem.
\end{theorem}
\subsection{Projection and interpolation errors with Jacobi polynomials}
Here we collect several results concerning projection and interpolation errors with respect to the weighted inner product (\ref{12a}) and that will be used below. We refer to, e.~g., \cite{GottliebO,Mercier,MadayQ1981,BernardiM1989,CanutoHQZ1988,CanutoQ1982a,BernardiM1997,ShenTW2011} for details and additional properties.

The estimates in the weighted Sobolev spaces $H_{w}^{s}, s\geq 0$, concern the use of the family of Jacobi polynomials $\{J_{n}^{\mu}\}_{n=0}^{\infty}$, which are orthogonal to each other in $L_{w}^{2}$. Particular cases such as Legendre and Chebyshev families correspond to $\mu=0$ and $\mu=-1/2$, respectively. Most properties of this Jacobi family (a particular one of the more general Jacobi polynomials $\{J_{n}^{\mu,\nu}\}_{n=0}^{\infty}$, orthogonal in $L_{w_{\mu,\nu}}^{2}$ with $w_{\mu,\nu}(x)=(1-x)^{\mu}(1+x)^{\nu}$) are extension of the corresponding properties of the Legendre family, \cite{BernardiM1989,BernardiM1997}.

We start with projection errors. Let $N\geq 2$ be an integer, $v\in H_{w}^{s}, s\geq 0$ and let $P_{N}v\in \mathbb{P}_{N}$ be the orthogonal projection of $v$ with respect to the inner product (\ref{12a}), and $P_{N}^{10}v\in \mathbb{P}_{N}^{0}$ be the orthogonal projection of $v$ with respect to the inner product in $H_{w,0}^{1}$
\begin{eqnarray*}
[\phi,\psi]_{w}=\int_{-1}^{1} \phi^{\prime}(x)\psi^{\prime}(x)w(x)dx.
\end{eqnarray*}
Then we have, \cite{BernardiM1989,CanutoQ1981,BernardiM1997}
\begin{eqnarray}
||v-P_{N}v||_{0,w}&\leq &C N^{-s}||v||_{s,w},\; v\in H_{w}^{s},\; s\geq 0,\label{c42a}\\
||v-P_{N}v||_{r,w}&\leq &C N^{2r-1/2-s}||v||_{s,w},\; v\in H_{w}^{s},\; 1\leq r\leq s,\nonumber
\end{eqnarray}
and for $v\in H_{w}^{s}\cap H_{w,0}^{1}$
\begin{eqnarray}
||v-P_{N}^{10}v||_{1,w}+N||v-P_{N}^{10}v||_{0,w}&\leq &C N^{1-s}||v||_{s,w},\nonumber\\
&&s\geq 1, -1<\mu\leq 0,\label{c42c}\\
||v-P_{N}^{10}v||_{1,w}+N^{1-\mu}||v-P_{N}^{10}v||_{0,w}&\leq &C N^{1-s}||v||_{s,w},\nonumber\\&& s\geq 1, 0<\mu\leq 1.\label{c42d}
\end{eqnarray}
In the Legendre and Chebyshev cases, sharper estimates hold, see \cite{CanutoQ1982a,CanutoHQZ1988}.

A third projection operator used below concerns the bilinear form $A$ given by (\ref{ad26}). If $v\in H_{w,0}^{1}$ then the orthogonal projection $\overline{v}\in \mathbb{P}_{N}^{0}$ of $v$ with respect to $A$ is defined as $\overline{v}=R_{N}v:(0,T)\rightarrow \mathbb{P}_{N}^{0}$ such that
\begin{eqnarray}
A(\overline{v}-v,\psi)=0,\; \psi\in \mathbb{P}_{N}^{0}.\label{ad27}
\end{eqnarray}
For this projection, we have, \cite{BernardiM1989}
\begin{eqnarray}
||v-\overline{v}||_{1,w}+N||v-\overline{v}||_{0,w}\leq C N^{1-m}||v||_{m,w},\label{ad27b}
\end{eqnarray}
for $v\in H_{w}^{m}, m\geq 1$. Furthermore, a generalized estimate can be obtained as follows. If $v\in H_{w}^{m}, m\geq 2$, let $u^{N}\in\mathbb{P}_{N}$ be a polynomial such that, \cite{CanutoHQZ1988}
\begin{eqnarray}
||u^{N}-v||_{k,w}\leq C N^{k-m}||v||_{m,w}.\; 0\leq k\leq 2.\label{ad27c}
\end{eqnarray}
By  using (\ref{ad27b}), (\ref{ad27c}) and the inverse inequalities, \cite{BernardiM1989}
\begin{eqnarray*}
||\psi||_{s,w}\leq CN^{2(s-r)}||\psi||_{r,w},\; \psi\in\mathbb{P}_{N},\; 0\leq r\leq s,\label{c41}
\end{eqnarray*}
we have
\begin{eqnarray}
||v-\overline{v}||_{2,w}&\leq & ||u^{N}-v||_{2,w}+||u^{N}-\overline{v}||_{2,w}\nonumber\\
&\leq & C N^{2-m}||v||_{m,w}+C N^{2}||u^{N}-\overline{v}||_{1,w}\nonumber\\
&\leq & C N^{2-m}||v||_{m,w}+C N^{2}\left(||u^{N}-v||_{1,w}+||v-\overline{v}||_{1,w}\right)\nonumber\\
&\leq & C N^{3-m}||v||_{m,w}.\label{ad27d}
\end{eqnarray}
%
%
Let $N\geq 2$ be an integer, $s\geq 0, v\in H_{w}^{s}$ and let $I_{N}v$ denote the interpolant polynomial of $v$ on $\mathbb{P}_{N}$ based on the Gauss-Lobatto-Jacobi nodes. 
The following estimates for the interpolation errors can be seen in \cite{BernardiM1989,BernardiM1997}: for $v\in H_{w}^{s}$
\begin{eqnarray}
||v-I_{N}v||_{r,w}&\leq &C N^{r-s}||v||_{s,w}\; v\in H_{w}^{s},\nonumber\\
&& 0\leq r\leq 1,\; s>\sup\{\frac{1+r}{2},\frac{1+r+\mu}{2}\}. \label{c43}
\end{eqnarray}
Finally, an additional estimate comparing the continuous and discrete inner products will be necessary:
if $f\in H_{w}^{1}$ and $\phi\in\mathbb{P}_{N}$, then
\begin{eqnarray}
|(f,\phi)_{w}-(f,\phi)_{N,w}|&\leq &C\left(||f-P_{N-1}f||_{0,w}\right.\nonumber\\
&&\left.+||f-I_{N}f||_{0,w}\right)||\phi||_{1,w},\label{c49}
\end{eqnarray}
where $(\cdot,\cdot)_{N,w}$ is given by (\ref{ad12}), \cite{CanutoHQZ1988}.

\section{Spectral Galerkin approximation}
\label{sec3}
Let $N\geq 2$ be an integer, $T>0$. The semidiscrete Galekin approximation is defined as the function $v^{N}:(0,T)\rightarrow \mathbb{P}_{N}^{0}$ satisfying, cf. \cite{AbreuD2020}
\begin{eqnarray}
A(v_{t}^{N},\psi)&=&B(v^{N},\psi),\; \psi\in\mathbb{P}_{N}^{0},\label{ad31a}\\
A(v^{N}(0),\psi)&=&A(v_{0},\psi),\; \psi\in\mathbb{P}_{N}^{0}.\label{ad31b}
\end{eqnarray}
\begin{remark}
\label{remark221}
In what follows, we will make use of the two identities below {(see, e.g., \cite{ArnoldDT1981,Quarteroni1987})}. Let $F=F(x,t,u)$ a $C^{1}$ function of $x,t,u$. Then
\begin{eqnarray}
F(v)-F(z)&=&(v-z)F^{*}(v,z),\label{ident1}\\
F(v)v_{x}-F(z)z_{x}&=&F(v)(v-z)_{x}+z_{x}(v-z)F^{*}(v,z),\label{ident2}\\
F^{*}(v,z)&=&\int_{0}^{1}F_{u}(v+\tau(z-v))d\tau .\nonumber
\end{eqnarray}
\end{remark}
Local existence and uniqueness of (\ref{ad31a}), (\ref{ad31b}) are ensured by standard theory of ordinary differential equations (ode) when (\ref{ad31a}) is considered as a finite system for the coefficients of $v^{N}$ in some basis of $\mathbb{P}_{N}^{0}$, by using the property of ellipticity of $A$ and continuity of $B$. Concerning this last point, some estimates on $B$ will be required in order to prove a global existence result. This is discussed in the following remark.

\begin{remark}
\label{remark21}
We write $B$ in (\ref{ad26}) in the form
\begin{eqnarray*}
B(\phi,\psi)&=&B_{1}(\phi,\psi)+B_{2}(\phi,\psi)+B_{3}(\phi,\psi)\\
B_{1}(\phi,\psi)&=&\int_{-1}^{1}\alpha(\phi)\phi_{x}(\psi w)_{x}dx,\\
B_{2}(\phi,\psi)&=&\int_{-1}^{1}\beta(\phi)\phi_{x}(\psi w)dx,\\
B_{3}(\phi,\psi)&=&\int_{-1}^{1}\gamma(\phi)(\psi w)dx,
\end{eqnarray*}
and assume that $\alpha, \beta, \gamma \in C_{b}^{1}(\Omega\times (0,T)\times \mathbb{R})$. Then, using continuity of (\ref{ad26a}), there are constants $\alpha_{1}, \beta_{1}>0$ such that
\begin{eqnarray*}
|B_{1}(\phi,\psi)|&\leq &\alpha_{1}||\phi||_{1,w}||\psi||_{1,w},\\
|B_{2}(\phi,\psi)|&\leq &\beta_{1}||\phi||_{1,w}||\psi||_{1,w}.
\end{eqnarray*}
As far as $B_{3}$ is concerned, from (\ref{ident1}), (\ref{ident2}) we can find a function $\delta=\delta(t)$, bounded on $(0,T)$, and a constant $\gamma_{1}$ such that
\begin{eqnarray*}
|B_{3}(\phi,\psi)|\leq (\delta+\gamma_{1}||\phi||_{1,w})||\psi||_{1,w}.
\end{eqnarray*}
\end{remark}

Global existence and uniqueness for (\ref{ad31a}), (\ref{ad31b}) and the convergence to the solution of (\ref{ad25}), (\ref{ad26}) are proved in the following result.
\begin{theorem}
\label{theorem32} For all $t\in (0,T)$, there is a unique solution $v^{N}(t)$ of (\ref{ad31a}), (\ref{ad31b}) satisfying
\begin{eqnarray}
||v^{N}||_{L^{\infty}(0,T,H_{w}^{1})}\leq C,\label{exist}
\end{eqnarray}
for some constant depending on $||v_{0}||_{1,w}$. Furthermore, 
assume that $v_{0}\in H_{w,0}^{m}, m\geq 1$, $a, c\in C^{m}(\Omega)\cap H_{w}^{m}$, $\alpha, \beta, \gamma\in C^{m}(\Omega\times (0,T)\times \mathbb{R})$ with $\alpha(\cdot,t), \beta(\cdot,t), \gamma(\cdot,t)\in H_{w}^{m}, t\in (0,T)$. Then
\begin{eqnarray}
||v^{N}-v||_{L^{\infty}(0,T,L_{w}^{2})}&\leq &C N^{-m},\label{ad37c}\\
||v^{N}-v||_{L^{\infty}(0,T,H_{w}^{1})}&\leq &C N^{1-m},\label{ad37b}
\end{eqnarray}
for some constant $C$ which depends on $||v_{0}||_{H_{w}^{m}}, \alpha, \beta, \gamma, T$ but not on $N$. 
\end{theorem}

\begin{proof}
Following previous approaches, \cite{ArnoldDT1981,Quarteroni1987}, we first assume that $\alpha, \beta, \gamma\in C_{b}^{1}(\Omega\times (0,T)\times \mathbb{R})$. By using property of ellipticity of $A$ and Remark \ref{remark21}, we set $\psi=v_{t}^{N}$ in (\ref{ad31a}) and have
\begin{eqnarray*}
||v_{t}^{N}||_{1,w}\leq C ||v^{N}||_{1,w}+\delta(t),
\end{eqnarray*}
for some constant $C$. Then
\begin{eqnarray}
||v^{N}||_{1,w}&=&||v^{N}(0)+\int_{0}^{t}v_{t}^{N}(s)ds||_{1,w}\nonumber\\
&\leq & ||v^{N}(0)||_{1,w}+C\int_{0}^{t}||v^{N}(s)||_{1,w}ds+||\delta||_{L^{1}(0,T)}.\label{ad_226}
\end{eqnarray}
From (\ref{ad31b})  with $\psi=v^{N}(0)$ and properties of continuity and coercivity of $A$ we have $||v^{N}(0)||_{1,w}\leq C||v_{0}||_{1,w}$. This and Gronwall's lemma applied to (\ref{ad_226}) imply the existence of $v^{N}(t)$ for all $t\in (0,T)$ and (\ref{exist}).

As far as the error estimates are concerned, let $\overline{v}$ be the projection defined in (\ref{ad27}) and 
\begin{eqnarray}
\eta=\overline{v}-v, e^{N}=v^{N}-v, \xi^{N}=\overline{v}-v^{N}=\eta-e^{N}\in\mathbb{P}_{N}^{0}.\label{var}
\end{eqnarray}
Note that, due to (\ref{ad27}), $A(\eta_{t},\psi)=0, \psi\in \mathbb{P}_{N}^{0}$ holds. Thus, (\ref{ad11a}) and (\ref{ad31a}) imply, for $\psi\in \mathbb{P}_{N}^{0}$
\begin{eqnarray}
A(\xi_{t}^{N},\psi)=-A(e_{t}^{N},\psi)=-(B(v^{N},\psi)-B(v,\psi)).\label{ad310}
\end{eqnarray}
The right hand side of (\ref{ad310}) is written as
\begin{eqnarray}
B(v^{N},\psi)-B(v,\psi)&=&\widetilde{B}_{1}+\widetilde{B}_{2}+\widetilde{B}_{3},\label{btilde}\\
\widetilde{B}_{1}&=&\int_{-1}^{1}\left(\alpha(v+e^{N})(v+e^{N})_{x}-\alpha(v)v_{x}\right)(\psi w)_{x}dx,\nonumber\\
\widetilde{B}_{2}&=&\int_{-1}^{1}\left(\beta(v+e^{N})(v+e^{N})_{x}-\beta(v)v_{x}\right)(\psi w)dx,\nonumber\\
\widetilde{B}_{3}&=&\int_{-1}^{1}\left(\gamma(v+e^{N})-\gamma(v)\right)(\psi w)dx.\nonumber
\end{eqnarray}
We use (\ref{ident1}), (\ref{ident2}), the hypothesis $\alpha, \beta, \gamma\in C_{b}^{1}$ and Theorem \ref{theorem21} to have
\begin{eqnarray*}
|\widetilde{B}_{1}|&=&|\int_{-1}^{1} (\alpha(v^{N})e_{x}^{N}+v_{x}e^{N}\alpha^{*}(v^{N},v))(\psi w)_{x}dx|\nonumber\\
&\leq &C||e^{N}||_{1,w}||\psi||_{1,w},\label{btilde1}\\
|\widetilde{B}_{2}|&=&|\int_{-1}^{1} (\beta(v^{N})e_{x}^{N}+v_{x}e^{N}\beta^{*}(v^{N},v))(\psi w)dx|\nonumber\\
&\leq &C||e^{N}||_{1,w}||\psi||_{1,w},\label{btilde2}\\
|\widetilde{B}_{3}|&=&|\int_{-1}^{1}e^{N}\gamma^{*}(v^{N},v))(\psi w)dx|\leq C||e^{N}||_{0,w}||\psi||_{0,w}.\label{btilde3}
\end{eqnarray*}
Therefore
\begin{eqnarray}
|B(v^{N},\psi)-B(v,\psi)|\leq  C||e^{N}||_{1,w}||\psi||_{1,w},\label{ad_228}
\end{eqnarray}
for $\psi\in\mathbb{P}_{N}^{0}$ and with $C$ depending on $||v_{0}||_{1,w}$. Then, taking $\psi=\xi_{t}^{N}$ and using the coercivity of $A$, (\ref{ad310}) and (\ref{ad_228}) we obtain
\begin{eqnarray}
||\xi_{t}^{N}||_{1,w}\leq C ||e^{N}||_{1,w}.\label{ad_229}
\end{eqnarray}
Since (\ref{ad31b}) implies that $v^{N}(0)=\overline{v}(0)$ and therefore
$\xi^{N}(0)=0$, then writing $e^{N}=\eta-\xi^{N}$ yields
\begin{eqnarray*}
||\xi^{N}(t)||_{1,w}&=&||\int_{0}^{t}\xi_{t}^{N}(s)ds||_{1,w}\nonumber\\
&\leq & C\int_{0}^{t}(||\xi^{N}(s)||_{1,w}+||\eta(s)||_{1,w})ds.\label{ad312}
\end{eqnarray*}
Therefore, (\ref{ad37b}) holds from Gronwall's lemma, the property $e^{N}=\eta-\xi^{N}$ and Theorem \ref{theorem21}.

We now prove the estimate (\ref{ad37c}). Lax-Milgram theorem, \cite{Evans}, ensures the existence of $\varphi\in H_{w,0}^{1}$ such that, \cite{BernardiM1989,MadayQ1981}
\begin{eqnarray}
A(\psi,\varphi)=(\xi_{t}^{N},\psi)_{0,w},\; \psi\in H_{w,0}^{1}.\label{ad91}
\end{eqnarray}
Actually ({ see \cite{CanutoHQZ1988}}) $\varphi\in H_{w}^{2}$ and
\begin{eqnarray}
||\varphi||_{2,w}\leq C ||\xi_{t}^{N}||_{0,w}.\label{ad92}
\end{eqnarray}

We take $\psi=\xi_{t}^{N}$ in (\ref{ad91}) and use (\ref{ad310}) to estimate, (cf. \cite{AbreuD2020})
\begin{eqnarray}
||\xi_{t}^{N}||_{0,w}^{2}&\leq & A(\xi_{t}^{N},\varphi)=A(\xi_{t}^{N},\varphi-P_{N}^{10}\varphi)+A(\xi_{t}^{N},P_{N}^{10}\varphi)\nonumber\\
&=&A(\xi_{t}^{N},\varphi-P_{N}^{10}\varphi)-B(e^{N},P_{N}^{10}\varphi)\nonumber\\
&=&A(\xi_{t}^{N},\varphi-P_{N}^{10}\varphi)+B(e^{N},\varphi-P_{N}^{10}\varphi)\nonumber\\
&&-B(e^{N},\varphi).\label{ad93}
\end{eqnarray}
Now, continuity of $A$, (\ref{c42c}), (\ref{c42d})  and (\ref{ad92}) imply
\begin{eqnarray}
A(\xi_{t}^{N},\varphi-P_{N}^{10}\varphi)&\leq & C||\xi_{t}^{N}||_{1,w}||\varphi-P_{N}^{10}\varphi||_{1,w}\nonumber\\
&\leq & C N^{-1}||\xi_{t}^{N}||_{1,w}||\varphi||_{2,w}\nonumber\\
&\leq & C N^{-1}||\xi_{t}^{N}||_{1,w}||\xi_{t}^{N}||_{0,w}.\label{ad94}
\end{eqnarray}
On the other hand, Remark \ref{remark21},  (\ref{c42c}), (\ref{c42d})  and (\ref{ad92}) lead to
\begin{eqnarray}
|B(e^{N},\varphi-P_{N}^{0}\varphi)|&\leq & C||e^{N}||_{1,w}||\varphi-P_{N}^{0}\varphi||_{1,w}\nonumber\\
&\leq & C N^{-1}||e^{N}||_{1,w}||\varphi||_{2,w}\nonumber\\
&\leq & C N^{-1}\left(||e^{N}||_{1,w}+\right.\nonumber\\
&&\left.||\delta-P_{N}^{10}\delta||_{0,w}\right)||\xi_{t}^{N}||_{0,w}.\label{ad95}
\end{eqnarray}
We now consider $\widetilde{B}_{j}, j=1,2,3$ defined in (\ref{btilde}). Integrating by parts, we can write, \cite{ArnoldDT1981,Quarteroni1987}
\begin{eqnarray*}
\widetilde{B}_{1}(e^{N},\varphi)&=&\int_{-1}^{1}e^{N}\left((-\alpha_{x}-\alpha_{u}(v^{N})v_{x}^{N}+v_{x}\alpha^{*})(\varphi w)_{x}\right.\\
&&\left.-\alpha(v^{N})(\varphi w)_{xx}\right)dx,
\end{eqnarray*}
and several applications of Hardy inequality, see e.~g. \cite{CanutoHQZ1988} (this is not necessary of course in the Legendre case $\mu=0$), hypothesis $\alpha \in C_{b}^{1}$, (\ref{exist}) and Theorem \ref{theorem21} imply
\begin{eqnarray}
|\widetilde{B}_{1}(e^{N},\varphi)|&\leq &C||e^{N}||_{0,w}||\varphi||_{2,w}.\label{ad_2215}
\end{eqnarray}
Similarly, we write
\begin{eqnarray*}
\widetilde{B}_{2}(e^{N},\varphi)=\int_{-1}^{1}e^{N}\left((-\beta_{x}-\beta_{u}(v^{N})v_{x}^{N}+v_{x}\beta^{*})(\varphi w)\right)dx,
\end{eqnarray*}
and hypothesis $\beta \in C_{b}^{1}$, (\ref{exist}), Theorem \ref{theorem21} and continuity of $L$ in (\ref{ad26a}) lead to
\begin{eqnarray}
|\widetilde{B}_{2}(e^{N},\varphi)|&\leq &C||e^{N}||_{0,w}||\varphi||_{1,w}.\label{ad_2216}
\end{eqnarray}
Finally, hypothesis $\gamma \in C_{b}^{1}$ implies
\begin{eqnarray}
|\widetilde{B}_{3}(e^{N},\varphi)|&\leq &C||e^{N}||_{0,w}||\varphi||_{0,w}.\label{ad_2217}
\end{eqnarray}
Thus (\ref{ad_2215})-(\ref{ad_2217}) along with (\ref{ad92}) yield
\begin{eqnarray}
|B(e^{N},\varphi)|\leq C||e^{N}||_{0,w}||\varphi||_{2,w}\leq C||e^{N}||_{0,w}||\xi_{t}^{N}||_{0,w}.\label{ad97}
\end{eqnarray}
If we apply (\ref{ad94}), (\ref{ad95}) and (\ref{ad97}) to (\ref{ad93}) we finally obtain
\begin{eqnarray*}
||\xi_{t}^{N}||_{0,w}^{2}&\leq & CN^{-1}\left(||\xi_{t}^{N}||_{1,w}+||e^{N}||_{1,w}\right)\\
&&+C\left(||e^{N}||_{0,w}+||\delta-P_{N}^{0}\delta||_{0,w}\right).
\end{eqnarray*}
If we use (\ref{ad_229}), (\ref{ad37b}), $e^{N}=\eta-\xi^{N}$, Gronwall's lemma and Theorem \ref{theorem21} then (\ref{ad37c}) holds. Finally, the proof is completed by observing that the hypothesis $\alpha, \beta, \gamma\in C_{b}^{1}$ can be removed by the same argument as in \cite{ArnoldDT1981,Quarteroni1987}.

\end{proof}

\section{Spectral collocation approximation}
\label{sec4}
Let $x_{j}, j=0,\ldots,N$ be the nodes corresponding to the Gauss-Lobatto quadrature associated to $w$ in (\ref{jacobiw}) and introduced above. We define the  semidiscrete collocation approximation as a mapping $v^{N}:(0,T)\rightarrow \mathbb{P}_{N}^{0}$ such that
\begin{eqnarray}
cv_{t}^{N}-(I_{N}(av_{xt}^{N}))_{x}=-(I_{N}(\alpha v_{x}^{N}))_{x}+\beta v_{x}^{N}+\gamma (v^{N}),\label{ad315a} 
\end{eqnarray}
at $x=x_{j}, j=1,\ldots,N-1$, with
\begin{eqnarray}
v^{N}(0)\big|_{x=x_{j}}=v_{0}(x_{j}),\; j=0,\ldots,N. \label{ad315b}
\end{eqnarray}
The same arguments as those of \cite{AbreuD2020} lead to the weak formulation of (\ref{ad315a}), (\ref{ad315b}) in the more general Jacobi case:
\begin{eqnarray}
A_{N}(v_{t}^{N},\psi)&=&B_{N}(v^{N},\psi),\; \psi\in\mathbb{P}_{N}^{0}\nonumber\\
v^{N}(0)&=&I_{N}v_{0},\label{ad317}
\end{eqnarray}
where, for $\phi, \psi\in\mathbb{P}_{N}^{0}$
\begin{eqnarray}
A_{N}(\phi,\psi)&=&(c\phi,\psi)_{N,w}+(a\phi_{x},w^{-1}(\psi w)_{x})_{N,w},\label{ad318a}\\
B_{N}(\phi,\psi)&=&(\alpha(\phi)\phi_{x},w^{-1}(\psi w)_{x})_{N,w}+(\beta(\phi)\phi_{x},\psi)_{N,w}\nonumber\\
&&+(\gamma(\phi),\psi)_{N,w}.\label{ad318}
\end{eqnarray}
From the equivalence with the bilinear form, \cite{BernardiM1989,BernardiM1997}
\begin{eqnarray*}
a_{N}(\phi,\psi)=(\phi,\psi)_{N,w}+(\phi_{x},w^{-1}(\psi w)_{x})_{N,w},\label{ad45b}
\end{eqnarray*}
which is continuous in $P_{N}\times P_{N}^{0}$ and coercive in $P_{N}^{0}$ for all weights $w_{\mu}, -1<\mu<1$, we have that $A_{N}$ in (\ref{ad318a}) satisfies the properties of continuity and ellipticity
\begin{eqnarray}
|A_{N}(\phi,\psi)|&\leq & C||\phi||_{1,N}||\psi||_{1,N},\; \phi\in P_{N}, \psi\in P_{N}^{0},\nonumber\\
A_{N}(\psi,\psi)&\geq &C||\psi||_{1,N}^{2},\; \psi\in P_{N}^{0},\label{ad_237}
\end{eqnarray}
where
\begin{eqnarray}
||\phi||_{1.N}^{2}=||\phi||_{N,w}^{2}+||\phi_{x}||_{N,w}^{2},\label{ad320}
\end{eqnarray}
\begin{remark}
\label{remark23}
Note also that if $\alpha\in C_{b}^{1}$, then the property of continuity of the bilinear form
\begin{eqnarray*}
(\phi,\psi)\mapsto (\phi_{x},w^{-1}(\psi w)_{x})_{N,w}, \;  \phi\in P_{N}, \psi\in P_{N}^{0},
\end{eqnarray*}
proved in \cite{BernardiM1997}, implies the continuity of the first term of $B_{N}$ in (\ref{ad318}). The other two terms can be estimated, when $\beta, \gamma\in C_{b}^{1}$, in a similar way, using the arguments of Remark \ref{remark21} and the equivalence of the norms $||\cdot||_{1,w}$ and (\ref{ad320}) in $\mathbb{P}_{N}$ given from Lemma \ref{lemma11}.
\end{remark}
\begin{lemma}
\label{lemma32}
There is a unique solution $v^{N}(t)$ of (\ref{ad317}) for all $t\in (0,T)$ with
\begin{eqnarray*}
||v^{N}||_{L^{\infty}(0,T,H_{w}^{1})}\leq C,\label{ad_239}
\end{eqnarray*}
where $C$ depends on $||v_{0}||_{H_{w}^{1}}$.
\end{lemma}
\begin{proof}
As before, let us first assume that $\alpha, \beta,\gamma\in C_{b}^{1}$. When (\ref{ad317})  is viewed as a finite ode system for the coefficients of $v^{N}$ with respect to some basis of $P_{N}^{0}$, standard theory proves local existence and uniqueness. In order to prove continuation for $t\in (0,T)$ of $v^{N}(t)$, previous comments on $B_{N}$ in Remark \ref{remark23} show that for $\psi\in P_{N}^{0}$
\begin{eqnarray}
|B_{N}(\phi,\psi)|\leq (C||v^{N}||_{1,N}+||\delta||_{N,w})||\psi||_{1,N},\label{ad_238}
\end{eqnarray}
for some constant $C$ and where $\delta$ is defined in Remark \ref{remark21}. Then, from (\ref{ad_237}) and (\ref{ad_238}), taking $\psi=v_{t}^{N}$ in first equation of (\ref{ad317}) leads to
\begin{eqnarray*}
||v_{t}^{N}||_{1,N}\leq C\left(||v^{N}||_{1,N}+||\delta(\cdot,t)||_{1,N}\right).
\end{eqnarray*}
The equivalence of the norms from Lemma \ref{lemma11} in $\mathbb{P}_{N}$ yields a similar inequality
\begin{eqnarray*}
||v_{t}^{N}||_{1,w}\leq C\left(||v^{N}(t)||_{1,w}+||\delta(\cdot,t)||_{1,w}\right).
\end{eqnarray*}
from which
\begin{eqnarray*}
||v^{N}(t)||_{1,w}
&\leq & ||v^{N}(0)||_{1,N}+C\int_{0}^{t}C\left(||v^{N}(s)||_{1,w}+||\delta(\cdot,s)||_{1,w}\right)ds.
\end{eqnarray*}
We conclude the proof applying Gronwall's lemma and the stability of Gauss-Lobatto interpolation in the $H_{w}^{1}$ norm, that is, \cite{BernardiM1989,CanutoHQZ1988}
\begin{eqnarray*}
||I_{N}v_{0}||_{1,w}\leq ||v_{0}||_{1,w}.
\end{eqnarray*}
The hypothesis $\alpha, \beta,\gamma\in C_{b}^{1}$ can be finally removed as in \cite{ArnoldDT1981,Quarteroni1987}.
\end{proof}

The corresponding convergence result is as follows.
\begin{theorem}
\label{theorem41} Let $m\geq 2$ and assume the hypotheses of Theorem \ref{theorem32}. Then
\begin{eqnarray}
||v^{N}-v||_{L^{\infty}(0,T,H_{w}^{1})}\leq C N^{2-m},\label{t33}
\end{eqnarray}
for some constant $C$ which depends on $||v_{0}||_{H_{w}^{m}},\alpha, \beta, \gamma, T$ but not on $N$.
\end{theorem}
\begin{proof}
Let $\eta, e^{N}, \xi^{N}$ be as given in (\ref{var}). For $\psi\in\mathbb{P}_{N}^{0}$, since $A(\eta_{t},\psi)=0$, we can write
\begin{eqnarray}
A_{N}(\xi_{t}^{N},\psi)=A_{N}(\overline{v}_{t},\psi)-A(\overline{v}_{t},\psi)+B(v,\psi)-B_{N}(v^{N},\psi).\label{ad323}
\end{eqnarray}
Note first that
\begin{eqnarray}
|A_{N}(\overline{v}_{t},\psi)-A(\overline{v}_{t},\psi)|&\leq &|(c\overline{v}_{t},\psi)_{N,w}-(c\overline{v}_{t},\psi)_{0,w}|\nonumber\\
&&+|\int_{-1}^{1} (E-I_{N})(a\overline{v}_{tx})(\psi w)_{x}dx|.\label{ad325}
\end{eqnarray}
where $E$ denotes the identity operator. Since $\overline{v}_{t}=I_{N}\overline{v}_{t}$ then (\ref{c49}), applied to the first term on the right hand side of (\ref{ad325}), and hypothesis on $c$ imply that
\begin{eqnarray}
 |(c\overline{v}_{t},\psi)_{N,w}-(c\overline{v}_{t},\psi)_{0,w}|\leq C||(E-P_{N-1})\overline{v}_{t}||_{0,w}||\psi||_{0,w}.\label{ad326}
\end{eqnarray}
Now, (\ref{c42a}) and (\ref{ad27b}) lead to
\begin{eqnarray}
||(E-P_{N-1})\overline{v}_{t}||_{0,w}&\leq & ||(E-P_{N-1}){v}_{t}||_{0,w}+||(E-P_{N-1})(v_{t}-\overline{v}_{t})||_{0,w}\nonumber\\
&\leq & C(N-1)^{-m}||v_{t}||_{m,w}+C(N-1)^{-1}||v_{t}-\overline{v}_{t}||_{1,w}\nonumber\\
&\leq &C(N-1)^{-m}||v_{t}||_{m,w}.\label{ad326b}
\end{eqnarray}
As far as the second term on the right hand side of (\ref{ad325}) is concerned, from continuity of $L$ in (\ref{ad26a}) we have, for $\psi\in\mathbb{P}_{N}^{0}$
\begin{eqnarray}
|\int_{-1}^{1}(E-I_{N})(a\overline{v}_{tx})(\psi w)_{x}dx|\leq ||(E-I_{N})(a\overline{v}_{tx})||_{0,w}||\psi||_{1,w}.\label{ad327}
\end{eqnarray}
Now, from  (\ref{c43}), the hypothesis on $a$ and (\ref{ad27d}) we can write
\begin{eqnarray}
||(E-I_{N})(a\overline{v}_{tx})||_{0,w}&\leq & ||(E-I_{N})(a{v}_{tx})||_{0,w}\nonumber\\
&&+||(E-I_{N})(a(v_{tx}-\overline{v}_{tx}))||_{0,w}\nonumber\\
&\leq & CN^{1-m}||av_{tx}||_{m-1,w}+CN^{-1}||a(v_{tx}-\overline{v}_{tx})||_{1,w}\nonumber\\
&\leq &CN^{1-m}||v_{t}||_{m,w}+CN^{2-m}||v_{t}||_{m,w},\label{ad327b}
\end{eqnarray}
where $C$ depends on $||a||_{m-1,w}$. In the last inequality the property, \cite{BernardiM1989,MadayQ1981}
\begin{eqnarray*}
||uv||_{s,w}\leq C ||u||_{s,w}||v||_{s,w},\; u,v\in H_{w}^{s}, s\geq 0,
\end{eqnarray*}
was used.

We now consider the last two terms in (\ref{ad323}), written as
\begin{eqnarray}
B(v,\psi)-B_{N}(v^{N},\psi)&=&B(v,\psi)-B(\overline{v},\psi)+B(\overline{v},\psi)-B_{N}(\overline{v},\psi)\nonumber\\
&&+B_{N}(\overline{v},\psi)-B_{N}(v^{N},\psi),\label{ad_2314}
\end{eqnarray}
and estimate each couple of (\ref{ad_2314}). Assume first that $\alpha, \beta,\gamma\in C_{b}^{1}$. Similar arguments to those of (\ref{ad_228}) along with  (\ref{ad27b}) imply
\begin{eqnarray}
|B(v,\psi)-B(\overline{v},\psi)|&\leq &C||v-\overline{v}||_{1,w}||\psi||_{1,w}\nonumber\\
&\leq &C N^{1-m}||v||_{m,w}||\psi||_{1,w}.\label{ad328}
\end{eqnarray}
On the other hand
\begin{eqnarray}
B(\overline{v},\psi)-B_{N}(\overline{v},\psi)&= & 
 \int_{-1}^{1} (E-I_{N})(\alpha\overline{v}_{x})(\psi w)_{x}dx\nonumber\\
&&+(\beta \overline{v}_{x},\psi)_{0,w}-(\beta\overline{v}_{x},\psi)_{N,w}. \nonumber\\
&&+(\gamma(\overline{v}),\psi)_{0,w}-(\gamma(\overline{v}),\psi)_{N,w}\nonumber
\end{eqnarray}
As in (\ref{ad327}), (\ref{ad327b})
\begin{eqnarray*}
| \int_{-1}^{1} (E-I_{N})(\alpha\overline{v}_{x})(\psi w)_{x}dx|\leq 
CN^{2-m}||v_{t}||_{m,w}||\psi||_{1,w}.
\end{eqnarray*}
Next, using (\ref{c49})
\begin{eqnarray*}
|(\beta \overline{v}_{x},\psi)_{0,w}-(\beta\overline{v}_{x},\psi)_{N,w}|&\leq &
C\left(||(E-P_{N-1})\beta(\overline{v})\overline{v}_{x}||_{0,w}\right.\\
&&\left. ||(E-I_{N})\beta(\overline{v})\overline{v}_{x}||_{0,w}\right)||\psi||_{0,w},
\end{eqnarray*} 
and  $\beta\in C_{b}^{1}$ along with (\ref{ad27d}) lead to
\begin{eqnarray*}
||(E-P_{N-1})\beta(\overline{v})\overline{v}_{x}||_{0,w}&\leq & C(N-1)^{1-m}||v_{x}||_{m-1,w}\\
&&+C(N-1)^{-1}||\overline{v}_{x}-v_{x}||_{1,w}\\
&\leq & C(N-1)^{2-m}||v||_{m,w},
\end{eqnarray*}
while, similarly to (\ref{ad327b})
\begin{eqnarray*}
||(E-I_{N})\beta(\overline{v})\overline{v}_{x}||_{0,w}&\leq & CN^{2-m}||v||_{m,w}.
\end{eqnarray*}
Finally, (\ref{c49}) also implies
\begin{eqnarray*}
|(\gamma(\overline{v}),\psi)_{0,w}-(\gamma(\overline{v}),\psi)_{N,w}|&\leq &
C\left(||(E-P_{N-1})\gamma(\overline{v})||_{0,w}\right.\\
&&\left. ||(E-I_{N})\gamma(\overline{v})||_{0,w}\right)||\psi||_{0,w}.
\end{eqnarray*} 
Now
\begin{eqnarray*}
||(E-P_{N-1})\gamma(\overline{v})||_{0,w}&\leq & ||(E-P_{N-1})\gamma({v})||_{0,w}\\
&&+||(E-P_{N-1})(\gamma({v})-\gamma(\overline{v}))||_{0,w}.
\end{eqnarray*}
Using (\ref{ident1}) with $F=\gamma$ we have
\begin{eqnarray*}
\gamma(v)&=&\delta+v\gamma^{*}(v,0),\\
\gamma(\overline{v})&=&\gamma(v)+(\overline{v}-v)\gamma^{*}(\overline{v},v).
\end{eqnarray*}
Then from (\ref{c42a}) and (\ref{ad27b})  we obtain
\begin{eqnarray*}
||(E-P_{N-1})\gamma({v})||_{0,w}&\leq & CN^{1-m}(||\delta(t)||_{m-1,w}+||v||_{m,w}),\\
||(E-P_{N-1})(\gamma({v})-\gamma(\overline{v}))||_{0,w}&\leq & CN^{1-m}||v||_{m,w}.
\end{eqnarray*}
All this leads to
\begin{eqnarray}
|B(\overline{v},\psi)-B_{N}(\overline{v},\psi)|&\leq &\left(CN^{2-m}||v||_{m,w}\right.\nonumber\\
&&\left.+N^{1-m}||\delta(t)||_{m-1,w}\right) ||\psi||_{0,w}.\label{ad_2316}
\end{eqnarray}
Finally, Remark \ref{remark23} on the continuity of $B_{N}$ and Lemma \ref{lemma11} imply
\begin{eqnarray}
|B_{N}(\overline{v},\psi)-B_{N}(v^{N},\psi)|\leq C||e^{N}||_{1,w}||\psi||_{1,w}.\label{ad330}
\end{eqnarray}
We may now take $\psi=\xi_{t}^{N}$ in (\ref{ad323}), use the coercivity of $A_{N}$ and  (\ref{ad325})-(\ref{ad330}) to obtain an estimate of the form
\begin{eqnarray*}
||\xi_{t}^{N}||_{1,w}\leq CN^{2-m}(||\delta(\cdot,t)||_{m-1,w}+||v(t)||_{m,w}+||v_{t}(t)||_{m,w}+||e^{N}||_{1,w}).
\end{eqnarray*}
Then  Gronwall's lemma, the property $e^{N}=\eta-\xi^{N}$ and Theorem \ref{theorem21} conclude the proof of (\ref{t33}). The condition $\alpha, \beta,\gamma\in C_{b}^{1}$ can be removed as in previous results.

\end{proof}
\begin{remark}
\label{remark41}
From the previous proof we can observe that the difference with respect to that of Theorem \ref{theorem32} is the use of the generalized estimate (\ref{ad27d}). This forces the term $N^{2-m}$ in (\ref{t33}), instead of $N^{1-m}$ as in (\ref{ad37b}).
\end{remark}

\section{Concluding remarks}
\label{sec5}
This paper is concerned with the numerical approximation of ibvp of pseudo-parabolic type with Dirichlet boundary conditions by using spectral discretizations in space. The approach, initiated with the {insightful} computational study in \cite{AbreuD2020} is here complemented with the {rigorous} numerical analysis of the spectral Galerkin and collocation discretizations  based on Jacobi polynomials. Existence of numerical solution and error estimates when approximating (\ref{ad11a})-(\ref{ad11c}) are proved. The estimates indeed depend on the regularity of the data. In particular, they show spectral convergence in the smooth case, confirming the corresponding experiments in \cite{AbreuD2020}. On the other hand, the different exponent in the decay of the error (which depends on this regularity) between the estimates in the Galerkin and collocation schemes is due to the inverse inequalities (\ref{c41}) associated to the Jacobi polynomials, cf. the Fourier spectral case studied in \cite{Quarteroni1987}. { Some concrete perspectives for a future research are shared with \cite{AbreuD2020}, but with important differences}. In particular, we are interested in searching for {a sharp} improvement of the estimates in order to justify some hypotheses conjectured from the  representative set of numerical experiments presented and discussed in \cite{AbreuD2020}. {Nevertheless, we keep in mind the challeging task to preserve the numerical efficiency and robustness to the overall spectral approach being analyzed.}

\section*{Acknowledgements}
E. Abreu was partially supported by FAPESP 2019/20991-8,
CNPq 306385/2019-8 and PETROBRAS 2015/00398-0 and 2019/00538-7.

\end{document}